\documentclass[11pt,A4paper]{article}

\usepackage[left=32mm,right=32mm,top=30mm,bottom=32mm]{geometry}
\usepackage{labelfig}
\usepackage{epsfig}
\usepackage{epstopdf}

\usepackage{xfrac}

\usepackage[percent]{overpic}

\usepackage{color}
\usepackage{amsthm,amsmath,amssymb}
\usepackage{booktabs}
\usepackage{mathpazo}
\usepackage{microtype}
\usepackage{overpic}
\usepackage{bm}
\usepackage{sectsty}
\usepackage[
	pdftitle={PDFTitle},
	pdfauthor={Hugo Parlier},
	ocgcolorlinks,
	linkcolor=linkred,
	citecolor=linkred,
	urlcolor=linkblue]
{hyperref}

\usepackage{mathtools}
\DeclarePairedDelimiter{\ceil}{\lceil}{\rceil}

\definecolor{linkred}{RGB}{199,21,133}
\definecolor{linkblue}{RGB}{16, 78, 139}

\usepackage[hang,flushmargin]{footmisc}
\usepackage{enumitem}
\usepackage{titlesec}
	\titlespacing{\section}{0pt}{12pt}{0pt}
	\titlespacing{\subsection}{0pt}{6pt}{0pt}
	
\titlelabel{\thetitle.\quad}

\makeatletter 

\long\def\@footnotetext#1{%
\H@@footnotetext{%
\ifHy@nesting 
\hyper@@anchor{\@currentHref}{#1}%
\else 
\Hy@raisedlink{\hyper@@anchor{\@currentHref}{\relax}}#1%
\fi 
}}

\def\@footnotemark{%
\leavevmode 
\ifhmode\edef\@x@sf{\the\spacefactor}\nobreak\fi 
\H@refstepcounter{Hfootnote}%
\hyper@makecurrent{Hfootnote}%
\hyper@linkstart{link}{\@currentHref}%
\@makefnmark 
\hyper@linkend 
\ifhmode\spacefactor\@x@sf\fi 
\relax 
}%

\ifFN@multiplefootnote%
\renewcommand*\@footnotemark{%
\leavevmode 
\ifhmode 
\edef\@x@sf{\the\spacefactor}%
\FN@mf@check 
\nobreak 
\fi 
\H@refstepcounter{Hfootnote}%
\hyper@makecurrent{Hfootnote}%
\hyper@linkstart{link}{\@currentHref}%
\@makefnmark 
\hyper@linkend 
\ifFN@pp@towrite 
\FN@pp@writetemp 
\FN@pp@towritefalse 
\fi 
\FN@mf@prepare 
\ifhmode\spacefactor\@x@sf\fi 
\relax%
}%
\fi 

\makeatother 

\theoremstyle{plain}
\newtheorem{theorem}{Theorem}[section]

\newtheorem{lemma}[theorem]{Lemma}
\newtheorem{corollary}[theorem]{Corollary}

\theoremstyle{definition}

\newtheorem{remark}[theorem]{Remark}

\newcommand{\sys}{{\rm sys}}

\newcommand{\N}{{\mathbb N}}

\newcommand{\Z}{{\mathbb Z}}

\newcommand{\CC}{{\mathcal C}}

\newcommand{\G}{{\mathcal G}}

\sectionfont{\large \bfseries}
\subsectionfont{\normalsize}

\long\def\symbolfootnote[#1]#2{\begingroup%
\def\thefootnote{\fnsymbol{footnote}}\footnote[#1]{#2}\endgroup}

\def\blfootnote{\xdef\@thefnmark{}\@footnotetext}

\begin{document}

{\Large \bfseries Short closed geodesics with self-intersections}

{\large Viveka Erlandsson and Hugo Parlier\symbolfootnote[1]{\normalsize Research of both authors supported by Swiss National Science Foundation grant number PP00P2\textunderscore 153024 
}
}

{\bf Abstract.} 
Our main point of focus is the set of closed geodesics on hyperbolic surfaces. For any fixed integer $k$, we are interested in the set of all closed geodesics with at least $k$ (but possibly more) self-intersections. Among these, we consider those of minimal length and investigate their self-intersection numbers. We prove that their intersection numbers are upper bounded by a universal linear function in $k$ (which holds for any hyperbolic surface). Moreover, in the presence of cusps, we get bounds which imply that the self-intersection numbers behave asymptotically like $k$ for growing $k$.

\vspace{1cm}

\section{Introduction} \label{s:introduction}

Closed geodesics play an important part in describing the geometry and dynamics of hyperbolic surfaces and their moduli. In particular, the length spectrum of a hyperbolic surface is closely related to analytic problems on surfaces as it determines the spectrum of the Laplacian. Among the closed curves, the simple ones play a particular role and are related to geometric and topological problems on moduli spaces including the study of homeomorphism groups and metrics on Teichm\"uller space. 

Among all closed geodesics, the shortest one is somewhat special and is called the systole of the surface. Unless a hyperbolic surface $X$ (with non-trivial fundamental group of finite type) is homeomorphic to a thrice punctured sphere, its systole is a simple closed geodesic. With this in mind, we are interested in the following problem introduced and studied by Basmajian and Buser. Given a fixed integer $k$, we consider the set of closed geodesics of $X$ that self-intersect at least $k$ times. Since the length spectrum is discrete, among them there is one of minimal length, say $\gamma$. By definition, $\gamma$ self-intersects at least $k$ times. The question is to find an upper bound on the number of self-intersection points of $\gamma$. 

As mentioned before, for $k=0$, this is asking for the number of self-intersections of the systole of $X$ and so unless $X$ contains no simple closed geodesics, the answer is $0$. For $k=1$, Buser \cite[Theorem 4.2.4]{BuserBook} solved the problem by showing that among all non-simple closed geodesics of $X$, the shortest one has a single intersection point (it is a so-called {\it figure eight geodesic}). The proof is an involved cut and paste type argument based on the observation that a non-simple closed geodesic contains a simple loop as a subset. Perhaps surprisingly, as far as exact values go, there are no further results known. 

A general result, due to Basmajian \cite{BasmajianStable}, provides a first answer to the question in the case where $X$ is complete, finite area and finite type. He shows that there exists a constant (that can be made explicit) which depends on $k$ and the topology of $X$ (but not its geometry) such that the number of self-intersections of $\gamma$ is upper bounded by this constant. If one works out the explicit bound, the dependence on $k$ is exponential. The bound on the topology is used to bound the lengths of curves in a pair of pants decomposition via a theorem of Bers \cite{Bers}, quantified by Buser and others \cite{BalacheffParlier,BuserBook,ParlierShort}. For general surfaces (those not necessarily of finite area), the methods proposed by Basmajian provide a bound which this time depends on the geometry of the surface, and in particular on a bound on the length of curves in a pants decomposition.

Let $I_k(X)$ denote the maximum number of self-intersections of a shortest geodesic on $X$ with at least $k$ self-intersections. We prove the following:

\begin{theorem}\label{thm:universal}
Let $X$ be an orientable complete hyperbolic surface with non-abelian fundamental group. Then 
$$I_k(X)\leq31\sqrt{k+\frac{1}{4}}\left(16\sqrt{k+\frac{1}{4}}+1\right)$$
\end{theorem}

The two main features of our result is that the growth is linear in $k$ (for instance the upper bound is less that $600\, k$ for all $k\geq 2$) and that there is no dependence on the geometry or the topology of the surface. In particular, it holds for {\it any} hyperbolic surface where the question makes sense (meaning with non-abelian fundamental group, including infinite area or infinite type surfaces, although this is not our focus point). While the final result does not depend on the geometry of the surface, one of the main ideas of our proof is to use the specific geometry of the surface to find appropriate decompositions of candidate curves.

Although the proof is mostly self-contained, it is certainly inspired by a flurry of recent results \cite{AGPS, ChasRelations, ChasPhillips, ErlandssonSouto, SapirBounds, SapirLower} focused on understanding the relationship between self-intersection and the length of closed geodesics. One of the tools we do use is the upper bounds of Basmajian \cite{BasmajianUniversal, BasmajianShort} on the length of the shortest curve with at least $k$ self-intersections. We note that these length bounds can be used directly to find a linear upper bound on $I_k(X)$ but the bound depends on the geometry of $X$ (see Section \ref{sec:setup} for more details).

Basmajian also shows that there is a considerable difference in the length growth depending on whether surfaces have a cusp or not: the growth rate for closed surfaces is roughly $\sqrt{k}$ whereas it is $\log(k)$ if the surface has cusps. We are able to exploit that growth difference to prove an asymptopically optimal result for cusped surfaces.

\begin{theorem}\label{thm:cuspcase}
Let $X$ be an orientable complete finite type hyperbolic surface with at least one cusp. Then there exists constants $D(X),K(X)$, depending on $X$, such that
$$I_k(X)\leq k+D(X) \log (k)$$
for all $k>K(X)$.
\end{theorem}
Exactly where the constants $D(X)$ and $K(X)$ come from can be found in Section \ref{sec:cusp}. Unlike in the previous theorem, the bounds here depends on the geometry of $X$. Although we do not want to dwell on it here, the condition on $X$ being of finite type can be relaxed to there being a positive lower bound on the systole length of $X$.

Note that Theorem \ref{thm:cuspcase} implies that
$$\lim_{k\to\infty}\frac{I_k(X)}{k}=1$$
when $X$ has a cusp. We conjecture that the above limit is always equal to $1$, regardless of whether $X$ has a cusp or not, but our methods do not seem to extend easily to more general surfaces. 

Our proof of Theorem \ref{thm:cuspcase} requires a generalization of Basmajian's lower bounds on lengths \cite{BasmajianUniversal}. In particular, we need to be able to control the relationship between length and intersection in the $\varepsilon$-thick part of a surface (which we denote $X_T$). As our result may be of independent interest, we state it here.

\begin{theorem}\label{thm:thick}
For $\varepsilon \leq \frac{1}{2}$, the intersection $\gamma_T= \gamma \cap X_T$ satisfies
$$
\ell(\gamma_T) >\frac{\varepsilon}{12} {\sqrt{i(\gamma_T,\gamma_T)}}
$$
\end{theorem}

Note that a closed surface is $\varepsilon$-thick for sufficiently small $\varepsilon$, so we recuperate Theorem 1.1 from \cite{BasmajianUniversal} with a somewhat different proof. 

We end the introduction by addressing the very natural question of lower bounds on $I_k(X)$. By definition, $I_k(X)\geq k$ with equality for infinitely many $k$. In fact, it is not a priori obvious that equality does not hold for {\it all} $k\geq 1$. However, there is a heuristic argument, inspired by results from \cite{BasmajianUniversal}, for why this should not always be the case. We illustrate it with a pair of pants $P$, say with three cuff lengths of length $1$. The local behavior of a closed geodesic is to either loop around one of the three boundary curves, or to follow some trajectory in the middle portion of the pair of pants, for instance that of a figure eight geodesic. If a closed geodesic closely follows a figure eight geodesic $n$ times, this creates roughly $n^2$ self-intersection points. On the other hand, a curve that loops $n$ times around a cuff creates roughly $n$ self-intersection points. Now assume there is a minimal length curve realizing $I_k(P)$ that has exactly $k$ self-intersections. Suppose you want to modify it to get a candidate for $I_{k+k_0}(P)$ for some $k_0$ relatively small compared to $k$. Each loop around a boundary costs you roughly $1$ in length, but although this is less than taking an extra copy of a figure eight curve, you are only getting one extra intersection point per loop. Thus, in terms of length, it would be more efficient to take (quasi) copies of a figure eight to generate self-intersection points than by looping around a boundary. Making the above argument rigorous would require a more delicate analysis of curves in pairs of pants, very different in nature from the methods used in this paper, but nonetheless, we expect that
$$
\limsup_{k\to \infty} (I_k(X) -k) = \infty
$$
for any compact $X$.

\section{Closed curves and their lengths}\label{sec:setup}

\subsection{Setup and known results}

Let $X$ be an orientable complete hyperbolic surface with non-abelian fundamental group. Said differently, we ask that $X$ is not the hyperbolic plane and is not topologically a cylinder. We want $X$ to have an interesting set of closed geodesics.

We will denote by $\G(X)$ the set of closed geodesics, by $\G_k(X)$ the subset of those that self-intersect exactly $k$ times, and by $\G_{\geq k}(X)$ those that intersect at least $k$ times. Basmajian studied the following quantity \cite{BasmajianUniversal, BasmajianShort}:
$$
s_k(X):= \inf \{ \ell(\gamma) : \gamma \in \G_k(X)\}
$$
showing that
$$
s_k(X) \leq 2 C_8(X) \sqrt{k+\frac{1}{4}}
$$
where $C_8(X)$ is the length of the shortest figure eight closed geodesic on $X$. (As mentioned above, Buser showed that $C_8(X)$ is also the length of the shortest non-simple closed geodesic of $X$.) The general gist of the proof of the above inequality is to construct a closed geodesic which follows the figure eight curve multiple times. The number of self-intersections of such a curve is roughly the square of the number of copies of the figure eight curve. To create a primitive closed curve, and to get the correct intersection number on the nose, require more delicate arguments. We remark that the above bound, from \cite{BasmajianShort}, is an improvement on previous bounds in \cite{BasmajianUniversal} where lower bounds on $s_k(X)$ are also explored. A fact about $s_k(X)$ that we will use in the sequel is the discrepancy between the growths when $X$ has cusps or not. The growth is logarithmic in $k$ when $X$ has a cusp. 

By discreteness of the length spectrum (for finite type surfaces), the value $s_k(X)$ is realized by the length of at least one closed geodesic. In particular, for $k=0$ this is the systole which, unless $X$ is a three holed sphere, is realized by a simple closed curve since the shortest non-trivial curve is always simple. If $X$ is a three holed sphere, the systole is a figure eight geodesic.

A related quantity is the following:
$$
s_{\geq k}(X):= \inf \{ \ell(\gamma) : \gamma \in \G_{\geq k}(X)\}
$$
and again it must be realized by the length of certain closed geodesics which may or may not have $k$ self-intersections. The actual number of self-intersections is our main concern in this article, and we will denote this number by $I_k(X)$. As $s_{\geq k}(X) \leq s_k(X)$, the inequality stated above for $s_k(X)$ also holds for $s_{\geq k} (X)$.

When $X$ is compact the upper bounds on $s_{\geq k} (X)$ are matched by lower bounds  \cite{BasmajianUniversal} of the form $C(X)\sqrt{k}$. Here the constant depends on the geometry of $X$ in such a way that $C(X)$ tends to $0$ when $X$ approaches the boundary of moduli space. These bounds, when appropriately put together, give a linear upper bound on $I_k(X)$ of type $U(X) k$ but where $U(X)$ this time goes to infinity as $X$ approaches the boundary of moduli space. In constrast, Basmajian's upper bounds \cite{BasmajianStable} on $I_k(X)$, when $X$ is complete and of finite area, only depend on the topology of $X$:
$$
I_k(X) \leq F(g, n, k)
$$
Here $g$ is the genus of $X$, $n$ the number of cusps and $F$ an explicit function. The proof is based on a generalization of the classical collar lemma for simple closed geodesics to closed geodesics. This generalized collar lemma implies that (self-)intersection points must create length, and as there is a bound on the length of the shortest curves with given lower bound on number of self-intersections, there cannot be arbitrarily many self-intersection points.

\subsection{Intersections and length}

We begin with the following lemma which relates lengths of simple closed geodesics and lengths of figure eight geodesics. 

\begin{lemma}\label{lem:simple8}
Let $\alpha,\beta$ be simple closed geodesics on $X$ with $i(\alpha,\beta) =1$ and $\ell(\alpha), \ell(\beta) \leq L$. Then
$$
C_8(X) < 4 L
$$
\end{lemma}
\begin{proof}
We think of $\alpha$ and $\beta$ as oriented loops based in their intersection point. The geodesic in the homotopy class of the closed curve obtained by the following concatenations
$$
\alpha * \beta * \alpha^{-1} * \beta
$$
is a figure eight geodesic whose length is strictly less than $2 \ell(\alpha) + 2\ell( \beta)$ which is at most $4L$. 
\end{proof}

As a corollary we have the following.

\begin{corollary}\label{cor:balls}
For any $p \in X$ and for all $r_0\leq \frac{C_8(X)}{8}$, the set $B_{r_0} (p)$ is topologically either a disk or a cylinder.
\end{corollary}

\begin{proof}
If not, then there is a point $p$ which is the base point of at least two distinct (and thus non-homotopic) simple geodesic loops $\alpha$ and $\beta$ of length at most $2 r_0$. These two loops could generate a pair of pants in which case the geodesic in the homotopy class of $\alpha*\beta$ is a figure eight geodesic of length at most $4 r_0 \leq \frac{C_8}{2}$ which is impossible. Otherwise they generate a one-holed torus in which case we refer to the previous lemma to conclude that $C_8(X) < 8 \frac{C_8}{8}$, again a contradiction.
\end{proof}

The above observation will be crucial in the sequel.

\section{Bounding intersection numbers}

We can now turn our attention to the problem at hand, namely the proof of Theorem \ref{thm:universal}. For clarity of exposition, we suppose that $X$ is of finite type. What we really use is the discreteness of the length spectrum which may fail if $X$ is of infinite type. In Remark \ref{rem:infinite} below, we discuss how to adapt the argument to when $X$ has a non-discrete length spectrum. However, we insist on the fact that this is not our focus point and the remark can be ignored by the reader only interested in finite type surfaces.

Let $\gamma \in \G_{\geq k} (X)$ be of minimal length. We seek to find an upper bound on $i(\gamma,\gamma)$. Once and for all, set $r_0$ to be the quantity
$$
r_0:= \frac{C_8(X)}{8}
$$
We cut $\gamma$ into segments $c_1, c_2, \hdots, c_m$, all of length $r_0$ except possibly $c_m$ which may be shorter. Note that by Basmajian's inequality 
$$
\ell(\gamma) < 2 C_8(X) \sqrt{k+\frac{1}{4}} = 16 r_0 \sqrt{k+\frac{1}{4}}
$$
and as such
$$
m \leq \ceil[\Bigg]{16 \sqrt{k+\frac{1}{4}}} < 16 \sqrt{k+\frac{1}{4}} + 1
$$
\begin{remark}\label{rem:infinite}
When the length spectrum of $X$ is not discrete, we cannot guarantee that $\gamma$ of minimal length exists (see \cite{BasmajianKim} for results about infinite type surfaces with non-discrete length spectra). However, Basmajian's inequality above continues to hold as we will briefly explain. The inequality depends only on $C_8(X)$, which may or may not be realized by a figure eight geodesic on $X$. Suppose it is not. Then there is a sequence of figure eight geodesics whose lengths $L_i$ tend to $C_8(X)$. Thus, for each $i\in \N$, there is a geodesic $\gamma_i$ with self-intersection at least $k$ satisfying the inequality 
$$
\ell(\gamma_i) < 2 L_i \sqrt{k+\frac{1}{4}} 
$$
From this we can deduce the existence of a $\gamma$ with self-intersection $k$ such that 
$$
\ell(\gamma) \leq 2 C_8(X) \sqrt{k+\frac{1}{4}} = 16 r_0 \sqrt{k+\frac{1}{4}}
$$
The arguments presented in what follows can all be adapted to the the non-discrete case by suitably replacing a minimal length $\gamma$ by a curve $\gamma$ of length arbitrarily close to the infimum of lengths. However, for clarity, we will not continually refer to how to adapt the arguments in this more general setting in the sequel.  
\end{remark}

Note that due to our choice of $r_0$ and Corollary \ref{cor:balls}, any pair of intersecting segments $c_i, c_j$ (not necessarily distinct) all live in either disks or cylinders. 
If they live in a disk, then they are simple and can pairwise intersect at most once. We observe therefore that if all pairs of segments lived in disks, there would be an immediate upper bound on self-intersection given by 
$$
\frac{m^2}{2} - m
$$
Replacing $m$ with the upper bound in terms of $k$ proves the main theorem in this case, but of course we cannot a priori suppose this to be the case. 

We use the word \emph{strand} for a segment in a cylinder which has both its endpoints on the boundary of the cylinder. In general this is not always the case for our segments $c_i$, however, we will often extend segments to strands. By abuse of notation, we denote the strand also by $c_i$. 

If a segment $c_i$ lives in a cylinder $\CC$, it can be one of two types. Consider $\delta_{+}$ and $\delta_{-}$ the two boundary curves of $\CC$. If the strand $c_i$ intersects both $\delta_{+}$ and $\delta_{-}$ in its endpoints, it is a simple geodesic segment as there is no topology to create self-intersection. We refer to this type as a \emph{crossing} strand (an example is the leftmost strand in Figure \ref{fig:strands}).

\begin{figure}[h]
\leavevmode \SetLabels
\endSetLabels
\begin{center}
\AffixLabels{\centerline{\epsfig{file =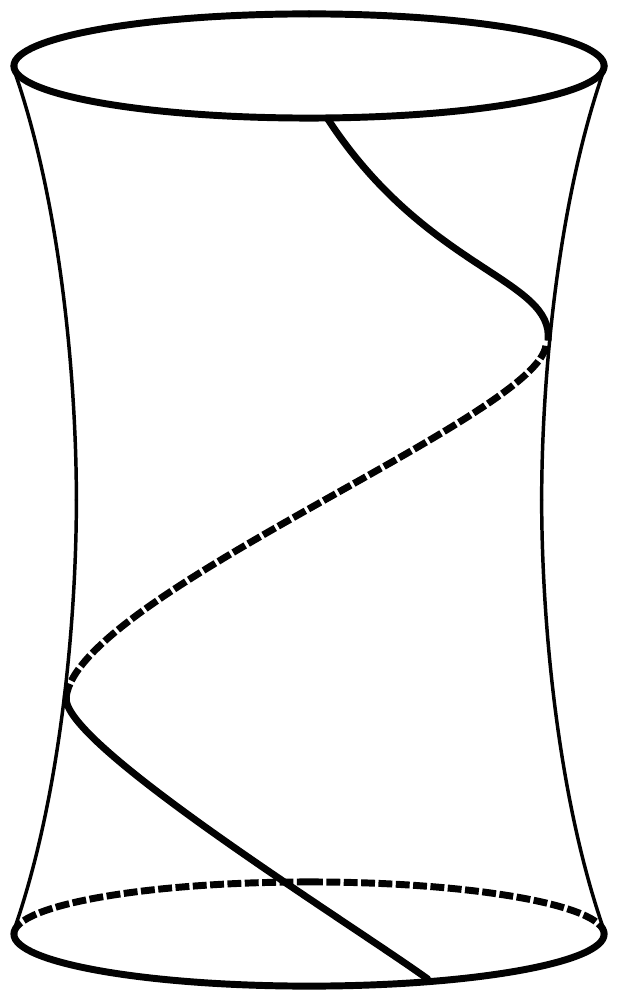,width=2.8cm,angle=0}\hspace{55pt}\epsfig{file =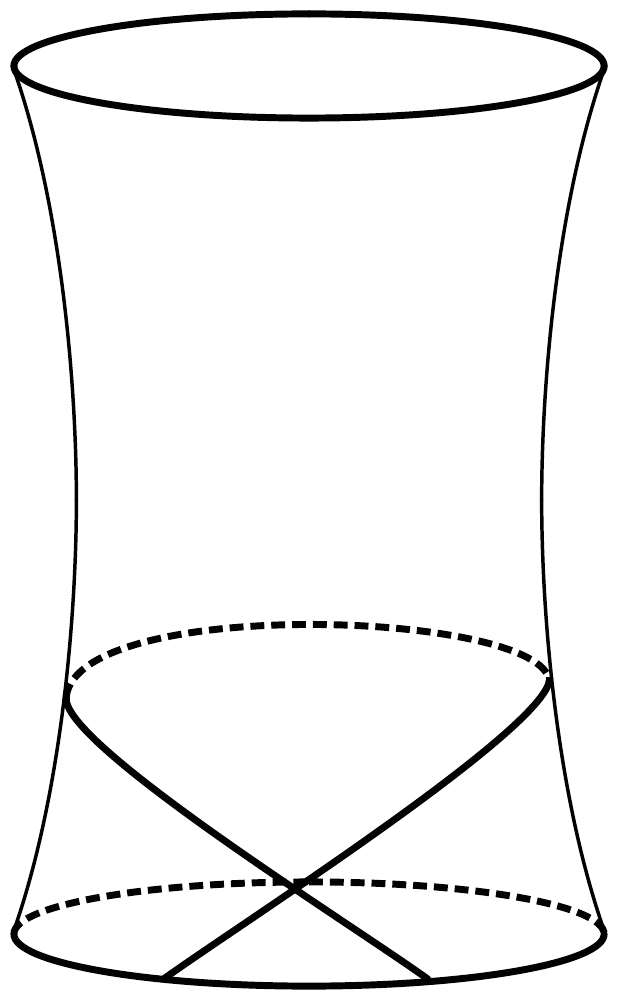,width=2.8cm,angle=0}\hspace{55pt}\epsfig{file =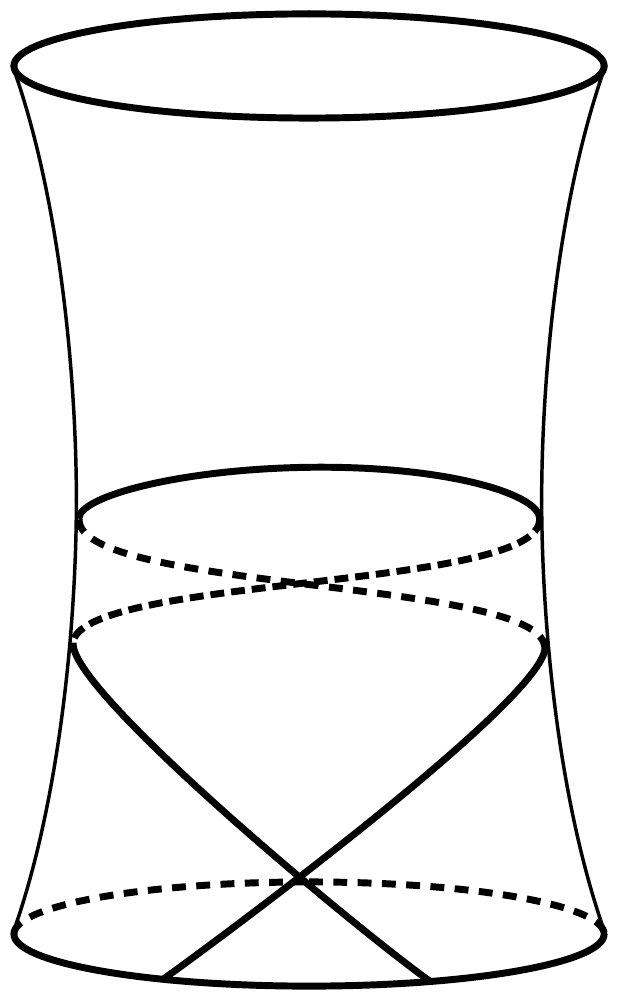,width=2.8cm,angle=0}}}
\vspace{-24pt}
\end{center}
\caption{A crossing strand and two returning strands} \label{fig:strands}
\end{figure}

The other type, which we will call a \emph{returning} strand, has both its endpoints on the same boundary curve, say $\delta_{-}$. In this case, it may have self intersection points which appear as a result of it wrapping around the core curve of the cylinder. In Figure \ref{fig:strands}, the middle and right strands have $1$ and $2$ self-intersection points. 

If the cylinder $\CC$ has core curve $\delta$ we define the \emph{winding number} $\omega(c_i)$ of a strand $c_i$ in $\CC$ (with respect to $\CC$) in the following way. Every point of $c_i$ projects to a well-defined point of $\delta$. The winding number of $c_i$ is given by the length of the projection of $c_i$ (thought of as a parameterized segment) divided by the length of $\delta$. 

Understanding the behavior of segments lying in embedded cylinders will be crucial. Here we record a fact about the intersection numbers of segments lying in the same cylinder.
\begin{lemma}\label{lem:int} 
Let $s_1, s_2$ be two distinct crossing strands, $r_1, r_2$ two distinct returning strands, all lying in the same cylinder, with $\omega(s_1)\leq\omega(s_2)$ and $\omega(r_1)\leq\omega(r_2)$. Then:
\begin{enumerate}
\item $i(s_1,r_1)\leq \lceil \omega(r_1) \rceil$
\item $i(r_1,r_1)\leq\lceil \omega(r_1) \rceil$
\item $i(r_1,r_2)\leq2\lceil\omega(r_1)\rceil$
\item $i(s_1,s_2)\leq\lceil\omega(s_1)\rceil$
\end{enumerate}  
\end{lemma}
 
 \begin{proof} 
 Suppose the cylinder $\CC$ has boundary curves $\delta_-$ and $\delta_+$ and core curve $\delta$. For each strand $c_i$ in $\CC$ we will construct a representative $c_i'$ homotopic to $c_i$ (relative its endpoints on $\delta_-$ and $\delta_+$) and use it to get an upper bound on the intersection numbers. Suppose $c_i$ has endpoints $p$ and $q$ on $\delta_-$ or $\delta_+$. Note that if $r_1$ has both its endpoints on $\delta_-$ and $r_2$ has both its endpoints on $\delta_+$ then $i(r_1,r_2)=0$. Hence we can assume with out loss of generality that $c_i$ has at least one endpoint on $\delta_-$. We construct $c_i'$ the following way. Choose a simple loop $\delta_{c_i}$ in the interior of $\CC$ such that every point is on it is equidistant to $\delta$. Let $c_i'$ be the curve consisting of the perpendicular segment between $p$ and $\delta_{c_i}$, a segment winding around $\delta_{c_i}$ according to $\omega(c_i)$, and finally the perpendicular segment between $\delta_{c_i}$ and $q$. Moreover, if $c_i$ is a returning strand, chose $\delta_{c_i}$ to be closer to $\delta_-$ than $\delta_+$, and if its a crossing strand chose it closer to $\delta_+$. Finally, if $c_i$ and $c_j$ are of the same type and $\omega(c_i)<\omega(c_j)$ choose $\delta_{c_i}$ to be closer to the boundary of $\CC$ than $\delta_{c_j}$ is (and when they have the same winding number, make an arbitrary choice). Clearly $c_i'$ is homotopic to $c_i$. 
 
 For $i=1,2$, let $s_i'$ and $r_i'$ be the representatives of $s_i$ and $r_i$ obtained as above. It is clear that $\vert s_1'\cap r_1'\vert\leq\lceil\omega(r_1)\rceil$ and since $i(s_1,r_1)\leq\vert s_1'\cap r_1'\vert$ we have proved the first part of the lemma. The remaining parts follow similarly.  
 \end{proof}

\subsection{Unwinding curves}

We begin by finding a bound on $i(\gamma,\gamma)$ in the case where a segment $c_i$ self-intersects more than $2$ times. Note that if this happens it necessarily lives inside a cylinder and is a returning strand.

\begin{lemma}\label{lem:unwind1}
If there exists $c_i$ with $i(c_i,c_i) \geq 2$, then 
$$
i(\gamma,\gamma) \leq  k -1 + 16  \sqrt{k+\frac{1}{4}}
$$
\end{lemma}

\begin{proof}
The segment $c_i$ contains a point of self-intersection $p$ and a geodesic simple loop based in $p$ as a subset.  This loop generates a cylinder $\CC$ of core geodesic $\delta$ (or possibly a cusp - in this case we set $\delta$ to be a small horocyclic neighborhood of the cusp disjoint and very far away from $c_i$). We observe that the parallel line $h_p$ to $\delta$ passing through $p$ is embedded in $X$ and moreover, the line parallel to $h_p$ consisting of points distance $r_0$ from $h$ is also embedded and is the boundary of an embedded cylinder. This is because otherwise there would be a point $p'$ with two geodesic loops of length at most $2 r_0$. As before, this would imply the existence of a figure eight geodesics of length strictly less than $C_8(X)$ which is not possible. 

We extend this cylinder maximally by boundary lines parallel to $\delta$ (both 'up' and  'down') and so that it remains embedded. The resulting cylinder we denote $\CC$ and we extend (if necessary) the segment $c_i$ so that both its endpoints lie on the other boundary curve of $\CC$ which we'll denote $\delta_{-}$. Note that $c_i$ is entirely contained in the half cylinder with boundary curves $\delta$ and $\delta_{-}$.

\begin{figure}[h]
{\color{linkred}
\leavevmode \SetLabels
\L(.525*.33) $h_p$\\%
\L(.465*.1) $c_i$\\%
\L(.425*.35) $p$\\%
\L(.398*.27) $r_0$\\
\L(.595*.97) $\delta$\\
\L(.49*-0.08) $\delta_{-}$\\
\endSetLabels
\begin{center}
\AffixLabels{\centerline{\epsfig{file =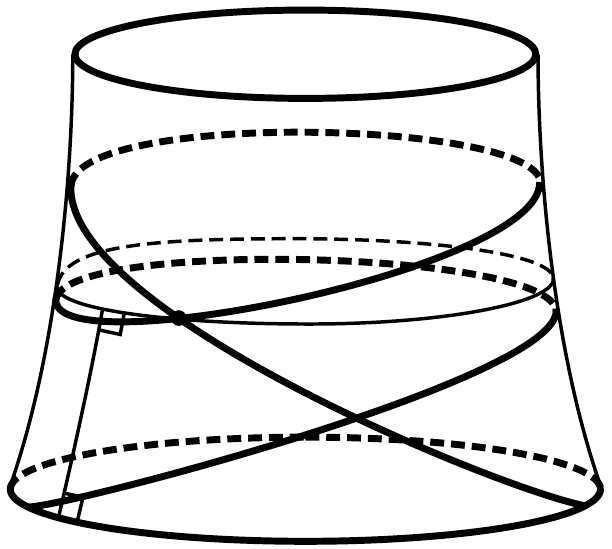,width=4.5cm,angle=0}}}
\vspace{-24pt}
\end{center}
}
\caption{The half cylinder containing $c_i$} \label{fig:height}
\end{figure}
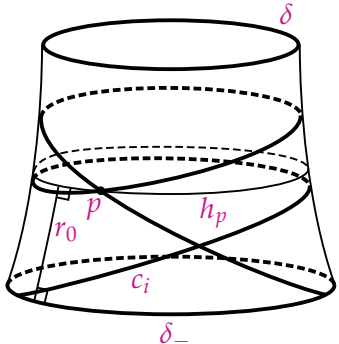

An important feature of this cylinder, which we will need below, is the following: Any geodesic arc $a$ which essentially crosses $\CC$ and has endpoints on $\partial \CC$, has length at least $2 r_0$. 

To see this consider a point $q$ which is the base point of a simple geodesic loop of length at most $2 r_0$ ($p$ is such a point). By repeating the argument above, the parallel line $h_q$ to $\delta$ at the level of $q$ is embedded in $\CC$, as is the cylinder consisting of all points at distance at most $r_0$ from $h_q$. In particular, the width of $\CC$ is at least $2 r_0$. 

Now consider an essential arc $a$ on $\CC$. If it is simple and goes across the cylinder it has length at least the width of the cylinder, thus at least $2r_0$. If it is non-simple with both endpoints on $\delta_{-}$, then it must have a point at distance at least $r_0$ from $\delta_{-}$ and so it must be of length at least $2r_0$. 

Because $i(c_i,c_i) \geq 2$, we have $w(c_i) \geq 2$. It will be convenient to think of $\CC$ as the quotient of its universal cover $\tilde{\CC}$ by the standard action of $\Z$ and look at copies of $c_i$ in this "unwrapped" version of $\CC$ (see Figure \ref{fig:lift}). 

\begin{figure}[h]
{\color{linkred}
\leavevmode \SetLabels
\L(.62*.28) $\tilde{c}_i$\\%
\L(.62*.88) $\tilde{\delta}$\\
\L(.49*-0.1) $\tilde{\delta}_{-}$\\
\endSetLabels
\begin{center}
\AffixLabels{\centerline{\epsfig{file =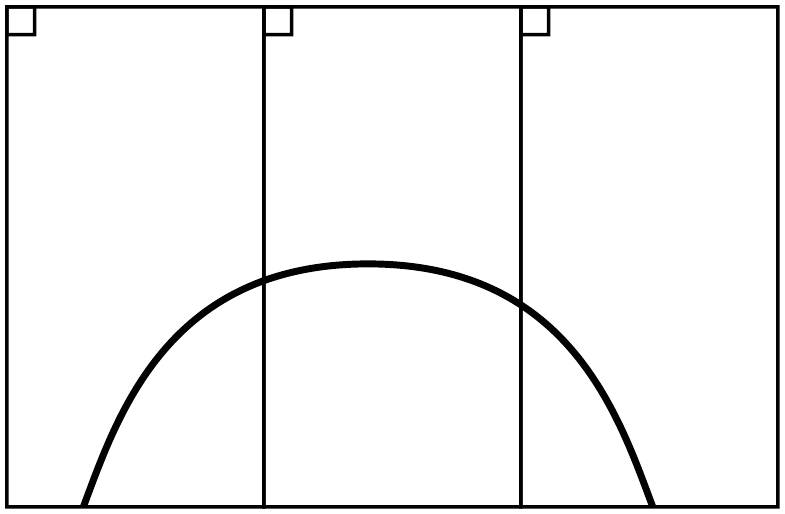,width=6.5cm,angle=0}}}
\vspace{-24pt}
\end{center}
}
\caption{A lift of the cylinder} \label{fig:lift}
\end{figure}

Let $c_i(t)$, $t\in [0,1]$ be a parametrization of $c_i$ and note that by standard hyperbolic geometry, the distance function $d_\CC(c_i(t), \delta)$ is strictly convex. (The function $d_\CC$ is the intrinsic distance function of $\CC$.)

Let $p$ be the closest self-intersection point of $c_i$ to $\delta$. It is the base point of a geodesic simple loop $\alpha$, which is a subset of $c_i$. We consider the closed geodesic $\gamma'$ in the homotopy class of the curve obtained from $\gamma$ by removing the loop $\alpha$ from $\gamma$. Note that necessarily $\ell(\gamma') < \ell(\gamma)$ and because of our choice of loop removal,  $\gamma'$ is not only non-trivial, we will be able to lower bound its self-intersection number. We begin by noting however that 
$$
i(\gamma',\gamma') \leq k-1
$$
otherwise $\gamma$ would not be minimal among elements of $\G_{\geq k}(X)$. 

To get a lower bound we will construct a representative of $\gamma$ from the geodesic $\gamma'$. Begin by observing that there is an arc of $\gamma'$ which lives on $\CC$ and which corresponds to the truncated strand $c_i$. 

\begin{figure}[h]
{\color{linkred}
\leavevmode \SetLabels
\L(.45*.45) $c'$\\%
\L(.49*.65) $p'$\\%
\L(.528*.505) $\alpha'$\\
\endSetLabels
\begin{center}
\AffixLabels{\centerline{\epsfig{file =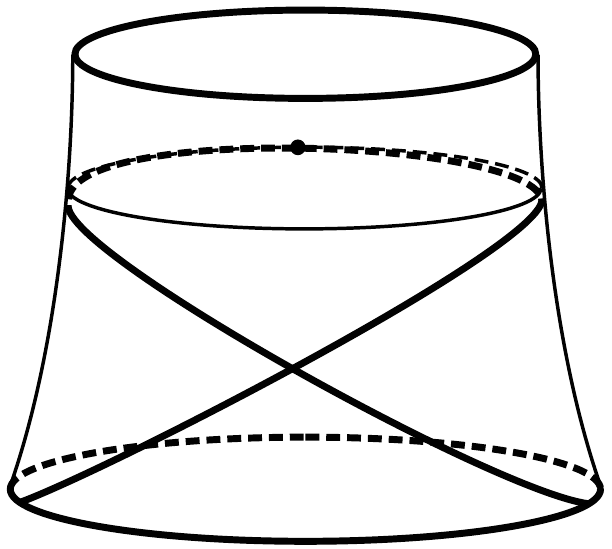,width=4.5cm,angle=0}}}
\vspace{-24pt}
\end{center}
}
\caption{The arcs $c'$ and $\alpha'$} \label{fig:height2}
\end{figure}

We'll denote it $c'$ and assume that it is oriented following some orientation of $\gamma'$. Consider its closest point $p'$ to $\delta$  and the loop $\alpha'$ formed by all points of $\CC$ of equal distance to $\delta$. Note that $\alpha'$ is freely homotopic to $\delta$ and thus to the loop $\alpha$ previously considered. We orient $\alpha'$ following the same orientation as $c'$. We consider the arc $c''$ obtained by following $c'$ from its orientation point until $p'$, then following $\alpha'$ and then continuing along $c'$. The important observation is that by replacing $c'$ with $c''$, we've recuperated the homotopy class of $\gamma$.

The number of self-intersection points of this representative of $\gamma$ is at least $i(\gamma,\gamma)$, but we'll be able to find an upper bound on this intersection number as well, which in turn will give us a bound on $i(\gamma,\gamma)$. 

We consider all the arcs of $\gamma'$ which are contained in the connected components of $\gamma' \cap \CC$ that might possibly intersect $\alpha'$. They must of course be essential strands that intersect $\CC$, and as observed above, must hence be of length at least $2r_0$. We can thus bound their number using our upper bound on the length of $\gamma'$. As
$$
\ell(\gamma') \leq 16 r_0 \sqrt{k+\frac{1}{4}}
$$
we have that the number of strands is at most 
$$
8 \sqrt{k+\frac{1}{4}}
$$
Because distance from points in $\CC$ to $\delta$ is strictly convex along parametrized geodesics, each strand can intersect $\alpha'$ at most twice. We thus have that 
$$
i(\alpha', \gamma') \leq 16  \sqrt{k+\frac{1}{4}}
$$
Therefore 
$$
i(\gamma,\gamma) \leq i(\gamma',\gamma') + \i(\alpha',\gamma') < k-1 + 16  \sqrt{k+\frac{1}{4}}
$$
as desired.
 \end{proof}

Observe that we can thus suppose in what follows that all of our segments are either simple or satisfy $i(c_i,c_i)=1$. A segment of the latter type we will call of $\alpha$-type, for obvious reasons. 

The same ``unwinding" technique from the proof of Lemma \ref{lem:unwind1} can be used to bound $i(\gamma,\gamma)$ when we have two simple arcs $c_i,c_j$ that intersect at least twice. First we need the following fact. 

\begin{lemma}
Suppose there exists a crossing strand $c_i$ lying in a cylinder $\CC$ with $\omega(c_i)>\frac{1}{2}$. Then 
$$i(\gamma,\gamma)\leq k-1+ 16  \sqrt{k+\frac{1}{4}}$$
\end{lemma}

\begin{proof}
Suppose $c_i$ lies in the cylinder $\CC$ with core curve $\delta$. We extend the cylinder maximally in parallel directions so that it remains embedded to obtain cylinder $C'$, still with core curve $\delta$. Note that the winding number of the corresponding strand $c_i$ still satisfies $\omega(c_i)>\frac{1}{2}$ with respect to $C'$. Also, by the same argument as in the proof of Lemma \ref{lem:unwind1} any geodesic arc that essentially crosses $C$ and has endpoints on $\partial C$ has length at least $2r_0$ and hence there are at most 
$$8\sqrt{k+\frac{1}{4}}$$ 
such strands. 

We now unwind $c_i$ once (by applying a single Dehn twist around $\delta$ to $c_i$, in such a way that the winding number of $c_i$ decreases). Let $\gamma'$ be the geodesic representative in the homotopy class of the resulting curve. Since $\omega(c_i)>1/2$ it follows that $\ell(\gamma')<\ell(\gamma)$ and hence, by the definition of $\gamma$, $i(\gamma',\gamma')\leq k-1$. 

\begin{figure}[h]
{\color{linkred}
\leavevmode \SetLabels
\L(.33*.78) $c_i$\\%
\L(.747*.78) $c'$\\
\endSetLabels
\begin{center}
\AffixLabels{\centerline{\epsfig{file =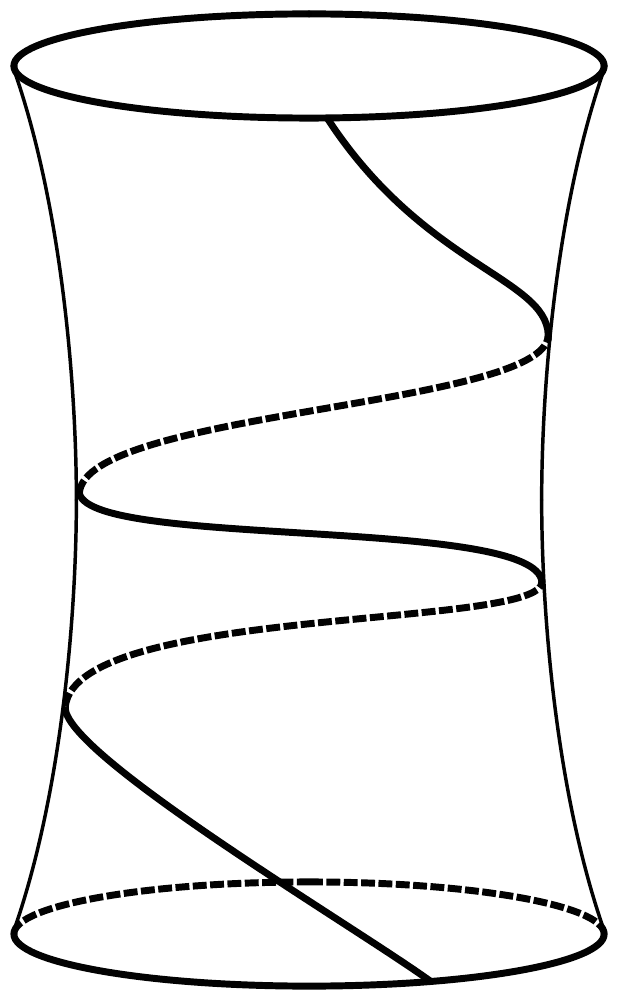,width=2.8cm,angle=0}\hspace{100pt}\epsfig{file =Crossing.pdf,width=2.8cm,angle=0}}}
\vspace{-24pt}
\end{center}
}
\caption{Unwinding a crossing strand} \label{fig:strands}
\end{figure}

We proceed in a manner similar to the proof of Lemma \ref{lem:unwind1}: we will reconstruct a representative of $\gamma$ from $\gamma'$ and use it to bound the self-intersection number of $\gamma$. Note that there is a strand $c'$ in a component of $\gamma'\cap\CC$ corresponding to $c_i$. Let $p$ be the intersection point between $c'$ and $\delta$. We choose some orientation of $\gamma'$ and orient $\delta$ in the 'winding' direction. Consider the arc $c''$ obtained by following $c'$ from one of its endpoints until $p$, then the loop $\delta$, and then continuing along $c'$ to its other endpoint. Let $\gamma''$ be the curve obtained from $\gamma'$ by replacing $c'$ with $c''$. Clearly $\gamma''$ is homotopic to $\gamma$ and hence $i(\gamma,\gamma)\leq i(\gamma'',\gamma'')$. By the exact same argument as in Lemma \ref{lem:unwind1} we have
$$i(\gamma,\gamma)\leq i(\gamma',\gamma')+i(\alpha,\gamma')\leq k-1+16\sqrt{k+\frac{1}{4}}$$
as desired. \end{proof}

If $c_i$ is a crossing strand in some cylinder $C$ with $\omega(c_i)\leq1/2$, it follows from Lemma \ref{lem:int} that it can intersect any other simple segment at most once. Hence we have: 

\begin{corollary}\label{cor:unwind2}
If there exists crossing strands $c_i,c_j$ with $i(c_i,c_j) \geq 2$, then 
$$i(\gamma,\gamma) \leq  k-1 + 16  \sqrt{k+\frac{1}{4}}$$\qed
\end{corollary}

  \subsection{$\alpha$-type segments and final estimates}
 
 We now place ourself in the situation where all of our segments are either simple or of $\alpha$-type. Furthermore, by Corollary \ref{cor:unwind2}, we can suppose that any two simple segments intersect at most once. 

We begin with a lemma about how an $\alpha$-type segment can intersect another segment:
 
 \begin{lemma}
 Let $c_i,c_j$ be two of our segments and suppose that $c_i$ is of $\alpha$-type. Then 
 $$
 i(c_i,c_j) \leq 4
 $$
 \end{lemma}
 
 \begin{proof}
 Since $c_i$ is $\alpha$-type we must have $\omega(c_i)\leq2$ (with respect to the cylinder for which it is $\alpha$-type). It follows from Lemma \ref{lem:int} that $i(c_i,c_j)\leq2$ if $c_j$ is simple and $i(c_i,c_j)\leq4$ if $c_j$ is $\alpha$-type. 
 \end{proof}
 
 We can now bound the intersection number of $\gamma$. Recall that the only cases left to consider are when $i(c_i,c_i)\leq 1$ and $i(c_i,c_j) \leq 4$ for all $i, j$. Hence we have: 
$$
i(\gamma,\gamma) \leq \frac{1}{2} \sum_{i,j=1}^m i(c_i,c_j) - \sum_{l=1}^m i(c_l,c_l) \leq 2 m^2 - m \leq 32 \sqrt{k+\frac{1}{4}} ( 16 \sqrt{k+\frac{1}{4}} +1)
$$
which proves the theorem. 

\section{Intersections in the thick part and surfaces with cusps}\label{sec:cusp}

The main goal of this section is to prove Theorem \ref{thm:cuspcase} but before doing so we study thick-thin decompositions of surfaces. 

\subsection{Thick parts of closed curves}

Given a hyperbolic surface $X$ and fixed $\varepsilon > 0$, we define the $\varepsilon$-thick part $X_T$ of $X$ to be the subset of $X$ consisting of points with injectivity radius at least $\varepsilon$. The $\varepsilon$-thin part $X_t$ is the subset of $X$ with injectivity radius at most $\varepsilon$. Now given a curve $\gamma \subset X$, we can decompose it into $\gamma_T := X_T \cap \gamma$ and $\gamma_t := X_t \cap \gamma$.

Note that $\gamma$ might go in and out of the thick part, so $\gamma_T$ is not necessarily the continuous image of an interval. Nonetheless $\gamma_T$ can be broken into arcs that are continuous images of intervals with endpoints lying on the boundary of the thick part and we will denote these components by $\gamma_1, \hdots, \gamma_r$. Our first observation is that, provided $\varepsilon$ is small enough, each $\gamma_i$ has a certain length.

\begin{lemma}
If $\varepsilon \leq \frac{1}{2}$, then 
$$
\ell(\gamma_i) \geq \frac{3}{4}
$$
for $i= 1,\hdots,r$
\end{lemma}

\begin{proof}
The boundary of $X_T$ consists of points of injectivity radius exactly $\varepsilon$ and, in particular, for any point of the boundary there is a simple geodesic loop of length $2 \varepsilon$ based in that point. Suppose that $\gamma_i$ joins points $p,q$ on the boundary of $X_T$ and denote by $\alpha$ and $\beta$ the simple loops of length $2 \varepsilon$ based in $p$ and $q$, respectively. Note that $\alpha$ and $\beta$ are either disjoint or freely homotopic. We orient $\gamma_i$, $\alpha$ and $\beta$ such that $\alpha$ and $\beta$ have opposite orientations. We now obtain a homotopy class of curve given by the concatenation
$$
\alpha * \gamma_i *\beta *\gamma_i
$$
The main observation is that the geodesic $\delta$ in the homotopy class of the above concatenation is a non-simple closed geodesic and thus has length at least $4 \log(1+\sqrt{2})$ (see for instance \cite{BuserBook}). Now as $\ell(\alpha) + \ell(\beta) + 2 \ell(\gamma_i)$ is a strict upper bound for $\ell(\delta)$, we have the inequality 
$$
2 \ell(\gamma_i) > 4 \log(1+\sqrt{2}) - 2 > \frac{3}{2}
$$
and the result follows.
\end{proof}
The constants in the above proof are clearly not optimal, and the choice of $\varepsilon \leq \frac{1}{2}$ is somewhat arbitrary. 

We now turn our attention to finding a lower bound on $\ell(\gamma_T)$ in terms of $i(\gamma_T, \gamma_T)$, proving Theorem \ref{thm:thick} of the introduction which gives a lower bound on length in terms of intersection number. 

\begin{proof}[Proof of Theorem \ref{thm:thick}]
We begin by considering a set of points $\{p_j\}_{j\in I}$ which form an $\varepsilon$-net for $X_T$ ($I$ is just an index set). Specifically, the points all belong to $X_T$, are pairwise at least distance $\varepsilon$ apart and are maximal for inclusion. In particular, any $x\in X_T$ is distance at most $\varepsilon$ from at least one $p_j$. As such we can consider the Voronoi cells $\{V_j \}_{j\in I}$ around each of the $p_i$. As $\varepsilon \leq \frac{1}{2}$, each of the Voronoi cells are (topological) disks. 

The intersection between $\gamma_T$ and any Voronoi cell $V_j$ is a collection of simple geodesic segments each of length at most $2 \varepsilon$. As $\gamma_T$ is of finite length, we can decompose $\gamma_T$ into these simple geodesic segments that traverse Voronoi cells. Denote them by $c_1, \hdots, c_m$. We note that an immediate upper bound on $i(\gamma_T, \gamma_T)$ is given by
$$
\frac{m(m-1)}{2}
$$
as any two of these segments can intersect at most once. We'll now proceed to bound $m$ in terms of $\ell(\gamma_T)$.

Recall our notation of $\gamma_1, \hdots, \gamma_r$ for the components of $\gamma_T$. By the previous lemma, we have $\ell(\gamma_i) \geq \frac{3}{4}$.

Consider the arc $\gamma_i$  consisting of multiple  $c_j$'s, the intersections with the Voronoi cells. We suppose the number of them is $m_i$ and we have
$$
\sum_{i=1}^{r} m_i = m
$$
We will now bound $m_i$ in terms of $\ell(\gamma_i)$. To do so we lift to the universal cover and consider the set of lifts of centers of Voronoi cells encountered by $\gamma_i$. We denote by $\tilde{\gamma}_i$ the lift of $\gamma_i$ and by $q_1,\hdots, q_{m_i}$ the lifts of the centers of the Voronoi cells. Note that the (open) balls of radius $\frac{\varepsilon}{2}$ around each $q_j$ are all pairwise disjoint. These balls are also all contained in the $\frac{3}{2} \varepsilon$ neighborhood of $\tilde{\gamma}_i$. The area of this neighborhood is obtained by computing the area of a strip of width $\frac{3}{2} \varepsilon$ around $\tilde{\gamma}_i$ and by adding the area of a ball of radius $ \frac{3}{2} \varepsilon$ for each of the two endpoints of $\tilde\gamma_i$. The resulting area is 
$$
A:= 2\left(\ell(\gamma_i) \sinh\sfrac{3\varepsilon}{2} + \pi(\cosh \sfrac{3\varepsilon}{2}-1)\right)
$$
In comparison, the total area of the balls of radius $\frac{\varepsilon}{2}$ around each $q_j$ is
$$
B:=m_i 2\pi (\cosh \sfrac{\varepsilon}{2}-1)
$$
and as $B < A$ we can deduce that 
$$
m_i < \frac{\ell(\gamma_i) \sinh\sfrac{3\varepsilon}{2} + \pi(\cosh \sfrac{3\varepsilon}{2}-1)}{\pi (\cosh \sfrac{\varepsilon}{2}-1)}
$$
We are not trying to optimize the constants we obtain, so we will simplify the above expression somewhat. Seen as a linear function in $\ell(\gamma_i)$, the leading coefficient can be bounded by
$$
\frac{\sinh\sfrac{3\varepsilon}{2} }{\pi (\cosh \sfrac{\varepsilon}{2}-1)}< \frac{5}{\varepsilon}
$$
as $\varepsilon \leq \frac{1}{2}$. The second coefficient is strictly increasing in $\varepsilon$ so, again using $\varepsilon \leq \frac{1}{2}$, we bound it by $10$. We thus have
$$
m_i < \frac{5}{\varepsilon} \ell(\gamma_i) + 10 < \frac{5}{\varepsilon} (\ell(\gamma_i)+1)
$$
Using the fact that $\ell(\gamma_i) > \frac{3}{4}$, this implies the following (highly non-optimal) inequality:
$$
m_i < \frac{12}{\varepsilon} \ell(\gamma_i)
$$
We now return to $\gamma_T$ and $m$. 

We have
\begin{eqnarray*}
i(\gamma_T,\gamma_T) & \leq & 
\frac{1}{2} m (m-1)\\
& =&  \frac{1}{2} \sum_{i=1}^r m_i \left(\sum_{i=1}^r m_i -1\right)\\
& < &\left( \frac{12}{\varepsilon}\right)^2 \left(\sum_{i=1}^r \ell(\gamma_i)\right)^2 = \left( \frac{12}{\varepsilon}\right)^2 \left(\ell(\gamma_T)\right)^2
\end{eqnarray*}
and thus
$$
\ell(\gamma_T) > \frac{\varepsilon \sqrt{i(\gamma_T,\gamma_T)}}{12}
$$
as desired.
\end{proof}

Note that if $X$ is closed, setting $\varepsilon := \min \{\frac{1}{2}, \frac{\sys(X)}{2}\}$ where $\sys(X)$ is the systole length of $X$, then $X=X_T$. In particular $\gamma$ is entirely contained in the thick part of $X$ and we have a lower bound on its length that grows like the root of its intersection. This is exactly the statement of Theorem 1.1 in \cite{BasmajianUniversal}. In what follows, we will need to apply our estimate to surfaces with cusps.

\subsection{Surfaces with cusps}

Armed with Theorem \ref{thm:thick} and using Basmajian's upper bounds on length for surfaces with cusps \cite{BasmajianUniversal}, we can now prove Theorem \ref{thm:cuspcase}. 

\begin{proof}[Proof of Theorem \ref{thm:cuspcase}]
Let $X$ be a complete hyperbolic surface with at least one cusp. If $\gamma$ is a closed geodesic on $X$ with at least $k\geq2$ self-intersections, it is a result by Basmajian \cite[Corollary 1.3]{BasmajianUniversal} that there exists a constant $C=C(k,X)$ such that $\ell(\gamma)<C$. In fact, $C=2\sinh^{-1}{(k)}+d_X+1$ where $d_X$ is the shortest orthogonal distance from the length 1 horosphere boundary of a cusp to itself. Note that $\sinh^{-1}(k)$ is comparable to $\log{(k)}$, and therefore so is $C(k,X)$. 

Let $\varepsilon'=\frac{1}{4}$ and let $s$ be the systole length of the $\varepsilon'$-thick part of $X$. Note that $\frac{1}{4}<\cosh^{-1}\left({\frac{\sqrt{11}}{3}}\right)$ which is the injectivity  radius of a cusp with boundary horosphere of length $\frac{2}{3}$. 

Now, let $\varepsilon=\min\left\{\frac{1}{4}, \frac{s}{2}\right\}$. Choose $K\geq2$ such that $C(k,X)<\frac{\varepsilon}{12}\sqrt{k}$ for all $k>K$. Let $k>K$ and $\gamma$ a shortest geodesic on $X$ with at least $k$ self-intersections. By Theorem \ref{thm:thick} $\gamma$ must intersect $X_t$, the $\varepsilon$-thin part of $X$. By the choice of $\varepsilon$, $\gamma$ must enter a cusp of $X$, and in fact a cusp neighborhood with boundary horosphere $\delta$ of length $\frac{2}{3}$.  This implies that $\gamma_t$ contains a strand $c$ (a continuous image of an interval with endpoints lying on the boundary horosphere) that intersects itself at least 3 times. We use a similar unwinding argument as in Lemma \ref{lem:unwind1} to get a bound on the intersection number of $\gamma$. Let $p$ be the self-intersection point of $c$ furthest away from $\delta$. It is the base point of a geodesic loop $\alpha$. Remove this loop from $\gamma$ and consider the resulting geodesic $\gamma'$. Clearly $\ell(\gamma')<\ell(\gamma)$ and hence, by definition of $\gamma$, $i(\gamma',\gamma')\leq k-1$. 

Let $c'$ be the strand of $\gamma'$ corresponding to the truncated strand $c$. Note that $c'$ self-intersects at least twice, and hence enters the cusp neighborhood (of the same cusp as $c$) with boundary horosphere $\delta'$ of length 1. Pick a point $p'$ on $c'$ in this cusp neighborhood and consider the simple loop $\alpha'$ based at this point consisting of all points equidistant from $\delta'$. As in Lemma \ref{lem:unwind1}, let $c''$ be the arc obtained by concatenating $c'$ and $\alpha'$ and let $\gamma''$ be the curve obtained by replacing $c'$ with $c''$ in $\gamma'$, and note that $\gamma''$ is homotopic to $\gamma$. Hence
$$i(\gamma,\gamma)\leq i(\gamma'',\gamma'')=i(\gamma',\gamma')+i(\alpha',\gamma').$$

To estimate $i(\alpha',\gamma')$ note that it is bounded from above by twice the number of strands of $\gamma'$ that enters the cusp neighborhood with boundary horosphere of length 1 (since each such strand can intersect $\alpha'$ at most twice). Each such strand has to pass through the cylinder of width $\log (2)$ in the cusp bounded by the horospheres of length 2 and 1, and then return. Hence each strand has length at least $2\log (2)$ and since $\ell(\gamma')<C(k,X)$ there are less than $C(k,X)/(2\log (2))$ such strands, and $i(\alpha',\gamma')<C(k,X)/\log (2)$.  Therefore, 
$$ i(\gamma,\gamma)<k-1+ \frac{C(k,X)}{\log (2)}$$
and, as noted above, $C(k,X)$ is comparable to $\log (k)$, proving the theorem.
\end{proof}

\addcontentsline{toc}{section}{References}
\bibliographystyle{plain}

\end{document}